\documentclass[11pt]{amsart}

\usepackage{amsfonts, amstext, amsmath, amsthm, amscd, amssymb, upgreek}
\usepackage[pdftex]{graphicx, color}
\usepackage[dvipsnames]{xcolor}
\usepackage[all,cmtip]{xy} 
\usepackage{tikz}
\usepackage{import}
\usepackage[hidelinks,pagebackref]{hyperref}
\usepackage{subfigure, wrapfig, overpic}

\let\oldmarginpar\marginpar
\renewcommand\marginpar[1]{\oldmarginpar[\raggedleft\footnotesize #1]%
{\raggedright\footnotesize #1}}

\usepackage{marginnote}
\long\def\@savemarbox#1#2{\global\setbox#1\vtop{\hsize\marginparwidth 
  \@parboxrestore\tiny\raggedright #2}}
\renewcommand*{\backref}[1]{}
\renewcommand*{\backrefalt}[4]{
  \ifcase #1
  [No citations.]
  \or [#2]
  \else [#2]
  \fi }

\AtBeginDocument{
   \def\MR#1{}
}

\theoremstyle{plain}
\newtheorem{theorem}{Theorem}[section]
\newtheorem{corollary}[theorem]{Corollary}
\newtheorem{lemma}[theorem]{Lemma}

\newtheorem{proposition}[theorem]{Proposition}

\newtheorem*{namedtheorem}{\theoremname}
\newcommand{\theoremname}{testing}

\theoremstyle{definition}

\newtheorem{definition}[theorem]{Definition}
\newtheorem{question}[theorem]{Question}

\newcommand{\refthm}[1]{Theorem~\ref{Thm:#1}}
\newcommand{\reflem}[1]{Lemma~\ref{Lem:#1}}
\newcommand{\refprop}[1]{Proposition~\ref{Prop:#1}}
\newcommand{\refcor}[1]{Corollary~\ref{Cor:#1}}

\newcommand{\refeqn}[1]{\eqref{Eqn:#1}}

\newcommand{\refdef}[1]{Definition~\ref{Def:#1}}
\newcommand{\refsec}[1]{Section~\ref{Sec:#1}}
\newcommand{\reffig}[1]{Figure~\ref{Fig:#1}}

\newcommand{\R}{\mathbb{R}}
\newcommand{\Z}{\mathbb{Z}}

\renewcommand{\H}{\mathcal H}
\newcommand{\vol}{{\rm vol}}

\newcommand{\voct}{{v_{\rm oct}}}
\newcommand{\vtet}{{v_{\rm tet}}}

\def\from{\colon\thinspace}

\title[A volumish theorem for alternating virtual links]{A volumish theorem for alternating virtual links}
\author[A.\ Champanerkar]{Abhijit Champanerkar}
\author[I.\ Kofman]{Ilya Kofman}
\address{Department of Mathematics, College of Staten Island \& The Graduate Center, City University of New York, New York, NY}
\email{abhijit@math.csi.cuny.edu, ikofman@math.csi.cuny.edu}

\begin{document}

\begin{abstract}
Dasbach and Lin proved a ``volumish theorem'' for alternating links.
We prove the analogue for alternating link diagrams on surfaces, which
provides bounds on the hyperbolic volume of a link in a thickened
surface in terms of coefficients of its reduced Jones-Krushkal
polynomial.  Along the way, we show that certain coefficients of the
$4$--variable Krushkal polynomial express the cycle rank of the
reduced Tait graph on the surface.
\end{abstract}

\maketitle

\section{Introduction}

In \cite{DL}, Dasbach and Lin proved the following ``volumish''
theorem for any hyperbolic alternating knot $K$ in $S^3$: Let
$$V_K(t)=a_nt^n+\cdots +a_mt^m$$ be the Jones polynomial of $K$, with
sub-extremal coefficients $a_{n+1}$ and $a_{m-1}$.  Let $\vtet\approx
1.01494$ and $\voct\approx 3.66386$ be the hyperbolic volumes of the
regular ideal tetrahedron and octahedron, respectively.  Then
$$ \voct(\max(|a_{n+1}|,|a_{m-1}|)-1)\leq \vol(S^3-K)\leq 10\vtet(|a_{n+1}|+|a_{m-1}|-1).$$
Their proof relied on volume bounds proved in \cite{Lackenby,
  AgolStormThurston}, which showed that the hyperbolic volume of
$S^3-K$ is linearly bounded above and below by the twist number
$t(K)$.  Dasbach and Lin proved that for any reduced alternating
diagram of $K$, the twist number $t(K)=|a_{n+1}|+|a_{m-1}|$.

Recently in \cite{HowiePurcell, KP}, similar linear volume bounds in terms of
twist number were proved for certain alternating links in thickened
surfaces, but the twist number was not proved to be a link invariant.
For alternating links in $S^3$, the invariance of $t(K)$ follows from
the proof of the Tait flyping conjecture in
\cite{menasco-thist:alternating}, but the Tait flyping conjecture
remains open for alternating virtual links (see \cite{Boden-Karimi}).

In \refsec{twist} below, for a link $K$ in a thickened surface
$F\times I$, we define a \emph{homological twist number} $\tau_F(K)$.
In \refsec{Jones-Krushkal}, we give a sufficient condition for
$\tau_F(K)$ to be an invariant of a reduced alternating surface link
diagram by expressing $\tau_F(K)$ in terms of specific coefficients of
the reduced Jones-Krushkal polynomial.  Using the new volume bounds in
terms of twist number, we prove a ``volumish'' theorem for alternating
links on surfaces, which extends to virtual links.

There is an underlying similarity between the proofs of the two
volumish theorems.  For alternating links in $S^3$, to prove that the
twist number is expressed by the sub-extremal coefficients of the
Jones polynomial, Dasbach and Lin relied on two key facts: (1) the
Jones polynomial of an alternating link is a specialization of the
two-variable Tutte polynomial of its Tait graph, and (2) certain
coefficients of the Tutte polynomial express the cycle rank of the
reduced Tait graph.  For alternating links in thickened surfaces, we
rely on two similar facts: (1) the reduced Jones-Krushkal polynomial
is a specialization of the Krushkal polynomial, which extends the
Tutte polynomial to a $4$--variable polynomial invariant of graphs on
surfaces, and (2) certain coefficients of the Krushkal polynomial
express the cycle rank of the reduced Tait graph on the surface (see
\refdef{Tait}).  The latter claim for the Krushkal polynomial is
\refthm{Krushkal}, which is of independent interest, and is proved in
Section~\ref{Sec:Krushkal} below.

Let $J_K(t,z)$ denote the reduced Jones-Krushkal polynomial, defined
in Section~\ref{Sec:Jones-Krushkal} below. Boden and Karimi
\cite{Boden-Karimi} proved that $J_K(t,z)$ is an invariant of oriented
links under isotopy and diffeomorphism of the thickened surface.
In \refthm{Jones-Krushkal}, we express the homological twist number in
terms of specific coefficients of $J_K(t,z)$.  This provides linear
bounds on the hyperbolic volume of the link $K$ in the thickened
surface in terms of the sub-extremal terms of $J_K(t,0)$ using the
following geometric results.

For a link $K$ in a thickened surface $F\times I$ with a \emph{weakly
  generalized alternating} (WGA) diagram, Howie and Purcell
\cite{HowiePurcell} defined the twist number $t_F(K)$ on the
projection surface $F\times\{0\}$, and showed there is a lower bound
on volume in terms of the twist number.  Note that if $F$ is a torus,
then $F\times I-K$ has a unique hyperbolic structure; for $g\geq 2$,
we consider the unique hyperbolic structure for which the boundary
surfaces $F\times\{\pm 1\}$ are totally geodesic.  A surface link
diagram $D$ is \emph{cellularly embedded} if the regions $F-D$ are
disks.  Kalfagianni and Purcell \cite{KP} proved there is also an
upper bound on volume when $K$ has a cellularly embedded WGA diagram
$D$.  In particular, $D$ has representativity at least $4$ on
$F$. (See \cite[Section~2]{KP} for definitions.)

A crossing $c$ is called \emph{nugatory} if there exists a separating
simple closed curve on $F$ that intersects $D$ only at $c$.  A surface
link diagram $D$ is called \emph{reduced} if it is cellularly embedded
and has no nugatory crossings.  Additionally, $D$ is \emph{strongly
  reduced} if there do not exist any simple closed curves on $F$ that
intersect $D$ at only one crossing; i.e., neither Tait graph of $D$ on
$F$ has loops.  A WGA diagram is reduced alternating, but it may not
be strongly reduced.

We now combine the hyperbolicity and lower bound from \cite{HowiePurcell}, the
upper bound from \cite{KP} modified for the homological twist number,
and our \refthm{Jones-Krushkal} below to state the volumish theorem
for alternating virtual links:

\begin{theorem}\label{Thm:Intro}
For a closed orientable surface $F$ of genus $g\geq 1$, let $K$ be
a non-split oriented link in $F\times I$ that admits a cellularly
embedded, strongly reduced WGA diagram $D$ on $F\times\{0\}$.
Let $\tau_F(K)$ be the homological twist number of $D$.
Let $J_K(t,0)=a_nt^n+\cdots +a_mt^m$, with sub-extremal coefficients $a_{n+1}$ and $a_{m-1}$.
Then 
$$ \tau_F(K) = |a_{n+1}| + |a_{m-1}| - 2g,$$
$\tau_F(K)$ is an invariant of $K$ in $F\times I$, and $F\times I-K$ is hyperbolic with
\begin{align*}
  \frac{\voct}{2} \tau_F(K) \leq \vol(F\times I-K) < 10\vtet\, \tau_F(K) &\quad\text{if} \ g=1,\\
  \frac{\voct}{2} (\tau_F(K)-3\chi(F)) \leq \vol(F\times I-K) < 12\voct\, \tau_F(K) &\quad\text{if} \ g\geq 2.
\end{align*}
\end{theorem}

We prove \refthm{Intro} in Section~\ref{Sec:Jones-Krushkal} below.
The strongly reduced condition on $D$ can be weakened to allow certain
loops in the Tait graph if we use the expression for $\tau_F(D)$ in
\refthm{Jones-Krushkal}.
See \refcor{Jones-Krushkal} for cases with loops such that
$\tau_F(D)$ is a link invariant.

\subsection*{Virtual links}

Virtual links and links in thickened surfaces are compared in detail
in \cite{Boden-Karimi}.  In short, virtual links are in one-to-one
correspondence with stable equivalence classes of links in thickened
surfaces, and each such class has a unique irreducible representative
\cite{Kuperberg}.  For any virtual link diagram, there is an explicit
construction to associate a cellularly embedded link diagram on a
minimal genus surface.  Moreover, a virtual link is alternating if and
only if it can be represented by an alternating surface link diagram.
Any reduced alternating surface link diagram is checkerboard
colorable, but alternating virtual links also admit alternating
surface diagrams which are not checkerboard colorable.  The main
result of \cite{Boden-Karimi} is the following diagrammatic
characterization of alternating links in thickened surfaces: If $K$ is
a non-split alternating link in $F\times I$, then any connected
reduced alternating diagram $D$ on $F$ has minimal crossing number
$c(K)$, and any two reduced alternating diagrams of $K$ have the same
writhe $w(K)$.  


The main result of \cite{Boden-Karimi} then implies that the reduced
alternating surface link diagram has crossing number and writhe that
are invariants of the virtual link.  By \cite[Corollary~8]{Boden-Karimi2020},
$F$ is the minimal genus representative of $K$.  So we obtain an
invariant of alternating virtual links by computing $J_K(t,z)$ on a
minimal genus representative reduced alternating surface link diagram
$D$.  The genus of $F$ is encoded as the highest power of $z$ in
$J_K(t,z)$.  \refcor{Jones-Krushkal} below then implies that the
homological twist number of $D$ on $F$ is also an invariant of the
virtual link.  Thus, \refthm{Intro} extends to any alternating virtual
link that admits an appropriate alternating surface link diagram.

\subsection*{Related results}
Recently, several preprints have appeared with related results.

In \cite{Boden-Karimi-Sikora}, Boden, Karimi and Sikora prove the
analogues of the Tait conjectures for adequate links in thickened
surfaces.  Any alternating link diagram in a thickened surface is
adequate, so a natural question is how to extend
\refthm{Jones-Krushkal} to adequate links in thickened surfaces.

In \cite{BaKaRu}, a general equivalence is established between ribbon
graphs and virtual links.  As our main results rely on the Krushkal
polynomial, which is an invariant of ribbon graphs, this philosophy
underlies our results as well.

In \cite{KalfBavier}, Bavier and Kalfagianni prove results similar to
\refthm{Intro} without using polynomial invariants of ribbon graphs.
Note that in \cite{KalfBavier}, \emph{reduced} is the same as
\emph{strongly reduced} here.  Their proof relies on the \emph{guts}
of a $3$--manifold cut along an essential surface, which is the union
of all hyperbolic pieces in its JSJ-decomposition, and the Euler
characteristic of the guts is related to the twist number using
results in \cite{Boden-Karimi-Sikora}.
Significantly, to prove that the twist number is invariant, Bavier and Kalfagianni
used another part of the Kauffman bracket skein module ${\mathcal
  S}(F\times I)$, which has a basis of all multi-loops on $F$,
including $\varnothing$.  Let $J_0(K)=b_nt^n+\cdots +b_mt^m$ be the
normalized invariant of $K$ in $F\times I$ coming from the coefficient
in $\Z[A^{\pm 1}]$ of $\varnothing$, so just the contractible states on
$F$.  They proved
  $ t_F(K) = |b_{n+1}| + |b_{m-1}| - 2g.$
In contrast, the Jones-Krushkal polynomial $J_K(t,0)$ uses states on $F$ that are null-homologous, including non-contractible states on $F$.
Thus, $J_K(t,0)\neq J_0(K)$ if $g\geq 2$, and in \refprop{twist} below, we show that $\tau_F(K)\neq t_F(K)$ if $g\geq 2$.
For links in thickened surfaces, we prove invariance of the homological twist number in
\refcor{Jones-Krushkal} for more general alternating link diagrams
than just strongly reduced ones because loops in Tait graphs are
allowed, as long as there are no genus-generating loops.


\subsection*{Acknowledgements}
The research of both authors is partially supported by grants from the Simons Foundation and PSC-CUNY.

\section{The Krushkal polynomial}
\label{Sec:Krushkal}

Krushkal \cite{Krushkal} introduced a $4$--variable polynomial
invariant of a graph $G$ embedded in a closed orientable surface $F$.
We denote this polyomial by $p_G(x,y,u,v)$ and refer to it as the
\emph{Krushkal polynomial}.  The variables $x$ and $y$ play the same
role as in the Tutte polynomial, while $u$ and $v$ reflect how $G$ is
embedded on $F$.  If $G$ is cellularly embedded (i.e., the faces of
$G$ on $F$ are disks), and $G^*$ denotes the dual graph on $F$, then
the Krushkal polynomial generalizes the Tutte polynomial, satisfying
both of its key properties: contraction-deletion and a duality
relation, $p_G(x,y,u,v)=p_{G^*}(y,x,v,u)$.

The Krushkal polynomial is defined as the following sum
over spanning subgraphs, such that every subgraph contributes a
monomial weight $x^ay^bu^cv^d$, where the exponents are topological
quantities related to the embedding of this subgraph. 

\begin{definition}[\cite{Krushkal}]\label{Def:Krushkal}
Let $G$ be a graph cellularly embedded in a closed orientable surface
$F$. The genus of a subsurface $S \subset F$ is the genus of the
closed surface obtained from $S$ by capping off all the boundary
components of $S$ by disks. For a spanning subgraph $H$ of $G$, let
$\mathcal{H}$ denote the regular neighborhood of $H$ on $F$.
Let $i\from G\to F$ denote the embedding, and let $i\from H\to F$ denote its restriction to $H$.
Define:
\begin{align*}
  c(H) &=  \text{number of components of}\ H, \\
  s(H)  &= \text{twice the genus of}\ \H, \\
  s^{\perp}(H) &= \text{twice the genus of the subsurface}\ F - \H, \\
  k(H) &= \dim(\ker(i_*\from H_1(H;\R)\to H_1(F;\R))).\\
\end{align*}
The Krushkal polynomial is defined as the following sum over all
spanning subgraphs $H\subset G$:
\begin{equation}
  \label{Eqn:Krushkal}
  p_G(x,y,u,v)=\sum_{H\subset G}x^{c(H)-c(G)}y^{k(H)}u^{s(H)/2}v^{s^{\perp}(H)/2}.
\end{equation}
We will refer to the monomial terms in \refeqn{Krushkal} as \emph{weights} on
corresponding subgraphs of $G$.
\end{definition}

The Tutte polynomial $T_G(X,Y)$ is related to the Whitney rank
generating function $R_G(x,y)$ by $T_G(X,Y)=R_G(X-1,Y-1)$ (see
\cite[\textsection~15.4]{Welsh}), which are extensively studied
polynomial invariants of graphs and matroids.  If $g$ denotes the
genus of $F$, by \cite[Lemma~2.3]{Krushkal},
\begin{equation}
  \label{eq:k-Tutte}
  R_G(x,y)=y^g p_G(x,y,y,y^{-1}),\ {\rm and} \  T_G(X,Y)=R_G(X-1,Y-1)
\end{equation}
The substitution $x = X-1$ and $y= Y-1$ will play a key
role in the proof of \refthm{Krushkal}.
So we define $$P_G(X,Y,U,V)=p_G(X-1,Y-1,U,V).$$
Another specialization to obtain the Jones-Krushkal polynomial is discussed in
Section~\ref{Sec:Jones-Krushkal}.
 
\begin{definition}\label{Def:parallel}
Two edges in $G$ are \emph{parallel} if they are homologous on $F$.
Note that parallel non-loop edges connect the same vertices, but
parallel loops may be disjoint.  Let $G'$ denote the \emph{reduced
  graph} of $G$ obtained by deleting all but one edge in each set of
parallel edges in $G$, and deleting all homologically trivial loops,
such that the vertex set $V(G')=V(G)$.  Let $G=(V,E)$ and $G'=(V,E')$.
Let $\ell=\ell(G')$ denote the subgraph of loops in $G'$, and let
$G'-\ell=(V,E'-\ell)$.
Let
$$\mu =b_1(G'-\ell)=|E'-\ell|-|V|+c(G') \quad \text{and}\quad \lambda = b_1(\ell)=|\ell|.$$
Note that although $G'$ is not uniquely determined, $\mu$ and $\lambda$ are invariants of $G$.
\end{definition}

\begin{theorem}\label{Thm:Krushkal}
Let $G$ be a graph embedded in a surface $F$ of genus $g\geq 1$. Let
$\ell_0$ be the set of homologically trivial loops in $G$.
Let $k=|\ell_0|$ and $n=|V(G)|-c(G)$.
Then $P_G(X,Y,U,V)$ has the following coefficients:
$$ \mu\, V^gX^{n-1}Y^k + \lambda\, V^{g-1}X^nY^k.$$
\end{theorem}
\begin{proof}
By \cite[Lemma~2.2]{Krushkal}, $p_G(x,y,u,v)$ has the property
that if $e$ is a loop in $G$ which is trivial in $H_1(F)$, then $p_G =
(1+y)p_{G-e}$, so that $P_G = Y P_{G-e}$.  Thus, we only need to prove
the case $|\ell_0|=0$, so we will consider only loops in $G$ that are
non-trivial in $H_1(F)$.

The unique spanning subgraph $H_0$ of $G$ which consists of only
vertices and no edges has weight $v^gx^n$. Since any other subgraph
has a non-empty edge set, its weight has a lower exponent of $x$ (if
it has non-loop edges), or a lower exponent of $v$ (if it has
homologically non-trivial loops).  Thus, the term $v^gx^n$ occurs in
$p_G(x,y,u,v)$ with coefficient $1$.

Let $e'$ be a non-loop edge of $G'$, and let $\{e_1,\ldots, e_m\}$ be
the set of all edges of $G$ parallel to $e'$, which we call the edge
class of $e'$. For $1 \leq j\leq m$, let $H_j$ denote one of the
spanning subgraphs of $G$ which consists of $j$ edges from the edge
class of $e'$, and no other edges. The weight of each $H_j$
is $v^gx^{n-1}y^{j-1}$.  Summing over the weights of all such spanning
subgraphs $\{ H_j\subset G\}$, we get the following contribution to
$p_G(x,y,u,v)$:
\begin{equation}\label{Eqn:mu-sum}
  \displaystyle{\sum_{j=1}^m {m \choose j} v^gx^{n-1}y^{j-1} = \frac{v^gx^{n-1} }{y}\left(\sum_{j=1}^m {m \choose j}y^j\right) = \frac{v^gx^{n-1} }{y}((1+y)^m-1)}.
\end{equation}
Thus, for every non-loop edge $e'$ in $G'$, its edge class in $G$
contributes the expression \refeqn{mu-sum} to $p_G(x,y,u,v)$.

If $H$ is a spanning subgraph of $G$ with the factor $x^{n-1}$ in its weight,
then $c(H)=|V|-1$.  Hence, $H$ has the form of some $H_j$, possibly
with loops added.  If $H$ has any loops, then since the loops are
homologically non-trivial by assumption, the weight of $H$ has an
exponent of $v$ which is strictly less than $g$.  Thus, any term in
$p_G(x,y,u,v)$ with a $v^gx^{n-1}$ factor is contributed only by
the subgraphs $H_j$, so the term must be $v^gx^{n-1}y^{j-1}$ for $j\geq 1$.

Let's see how these terms transform in $P_G(X,Y,U,V)$.
With the substitution $x = X-1$ and $y= Y-1$,
the expression \refeqn{mu-sum} simplifies to 
\[ \frac{V^g(X-1)^{n-1} }{Y-1}(Y^m-1) = V^gX^{n-1}(1+Y^2 +\ldots + Y^{m-1}) + O(X^{n-2}). \]
Every non-loop edge in $G'$ contributes such an expression to $P_G(X,Y,U,V)$.
Moreover, as discussed above, the weight for $H_0$ is $v^gx^n$, which becomes $V^g(X-1)^n$.
Since $v^gx^n$ always has coefficient $1$ in $p_G$, $H_0$ contributes an additional coefficient $-n$ to the term $V^gX^{n-1}$ in $P_G$.
Therefore, if $|\ell_0|=0$, the coefficient on $V^gX^{n-1}$ in $P_G(X,Y,U,V)$ is
$$|E'-\ell|-n = |E'-\ell|-|V|+1 =b_1(G'-\ell)=\mu.$$
This proves the claim for $\mu$.

We now proceed similarly for loops in $G'$. Let $f'$ be a loop of
$G'$, and let $\{f_1,\ldots, f_m\}$ be the set of all loops of $G$
parallel to $f'$, which we call the edge class of $f'$.  For $1 \leq
j\leq m$, let $L_j$ denote one of the spanning subgraphs of $G$ which
consists of $j$ loops from the edge class of $f'$, and no other edges.
Since we assumed that all loops in $G$ are homologically non-trivial,
the weight of $L_j$ is $v^{g-1}x^{n}y^{j-1}$.  Summing over the
weights of all such spanning subgraphs $\{ L_j\subset G\}$, we get the
following contribution to $p_G(x,y,u,v)$:
\begin{equation}\label{Eqn:lambda-sum}
  \displaystyle{\sum_{j=1}^m {m \choose j} v^{g-1}x^{n}y^{j-1} = \frac{v^{g-1}x^{n} }{y}\left(\sum_{j=1}^m {m \choose j}y^j\right) = \frac{v^{g-1}x^{n} }{y}((1+y)^m-1)}.
\end{equation}
Thus, for every loop $f'$ in $G'$, its edge class in $G$
contributes the expression \refeqn{lambda-sum} to $p_G(x,y,u,v)$.

If $H$ is a spanning subgraph of $G$ with the factor $x^{n}$ in its
weight, then $c(H)=|V|$.  Hence, $H$ consists of only homologically
non-trivial loops.  We have three cases:
\begin{enumerate}
\item[(a)] All loops in $H$ are in one edge class of $G'$,  
\item[(b)] $H$ has loops in distinct edge classes of $G'$, and $g(\H)=0$,
\item[(c)] $H$ has loops in distinct edge classes of $G'$, and $g(\H)>0$.
\end{enumerate}
In case (a), $H$ is one of the subgraphs $L_j$.  In case (b), $H$ has
at least one pair of homologically non-trivial and non-homologous
loops, so $g(\H)=0$ implies that $F-\mathcal{H}$ has genus strictly
less than $g-1$. Hence, the weight of $H$ has an exponent of $v$ which
is strictly less than $g-1$.  In case (c), the weight of $H$ has a
factor $u^i$ with $i>0$.  Therefore, any term in $p_G(x,y,u,v)$ with a
$v^{g-1}x^n$ factor and without a $u$ factor is contributed only by
the subgraphs $L_j$, so the term must be $v^{g-1}x^{n}y^{j-1}$ for
$j\geq 1$.

With the substitution $x = X-1$ and $y= Y-1$,
the expression \refeqn{lambda-sum} simplifies to 
\begin{equation}\label{Eqn:lambda-sum2}
  \frac{V^{g-1}(X-1)^{n} }{Y-1}(Y^m-1) = V^{g-1}X^{n}(1+Y^2 +\ldots +Y^{m-1}) + O(X^{n-1}).
\end{equation}
Every loop in $G'$ contributes such an expression to
$P_G(X,Y,U,V)$, so if $|\ell_0|=0$, the coefficient on $V^{g-1}X^{n}$ in
$P_G(X,Y,U,V)$ is $\lambda$.
This completes the proof of the theorem.
\end{proof}

Below, we will need another coefficient of $P_G(X,Y,U,V)$, using the following definition.
\begin{definition}\label{Def:petal}
For a graph $G$ on the surface $F$, let $\ell(G)$ be the subgraph of loops in $G$.
We will say that $\{e_1,e_2\}\subset\ell(G)$ are \emph{genus-generating loops} if $g(\H(e_1\cup e_2))>0$.
Let $G'$ be the reduced graph of $G$.  Define
$$  \gamma(G) = \#\{\,\{e_1,e_2\}\subset\ell(G')\ |\ g(\H(e_1\cup e_2))>0\}.$$
We will say that $\{e_1,e_2,e_3\}\subset\ell(G)$ are \emph{$3$--petal loops} if no pair of loops is parallel and
$$g(\H(e_1\cup e_2\cup e_3))>0 \quad \text{and}\quad k(e_1\cup e_2\cup e_3)>0.$$
Note that if $\gamma(G)=0$, then $G$ has no $3$--petal loops.
The following figure shows an example of a graph with $3$--petal loops on the torus:
\begin{center}
  \includegraphics[height=55pt]{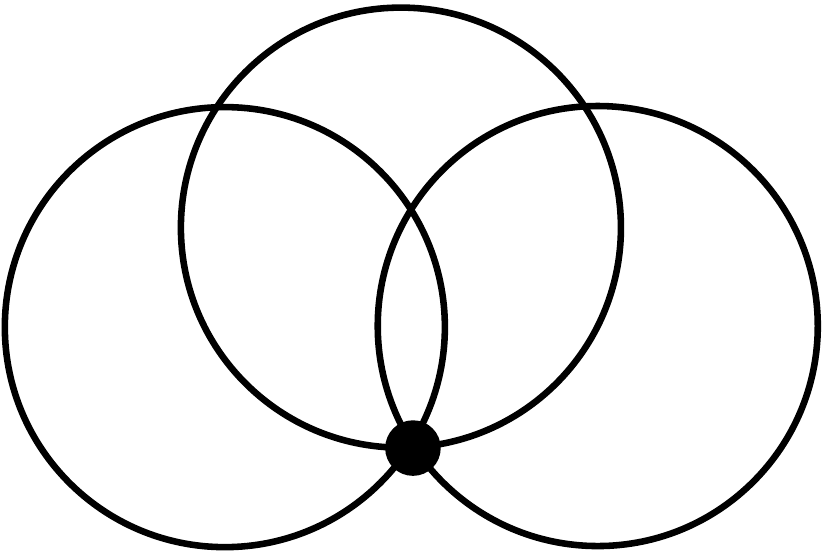}\hspace*{0.5in} \includegraphics[height=60pt]{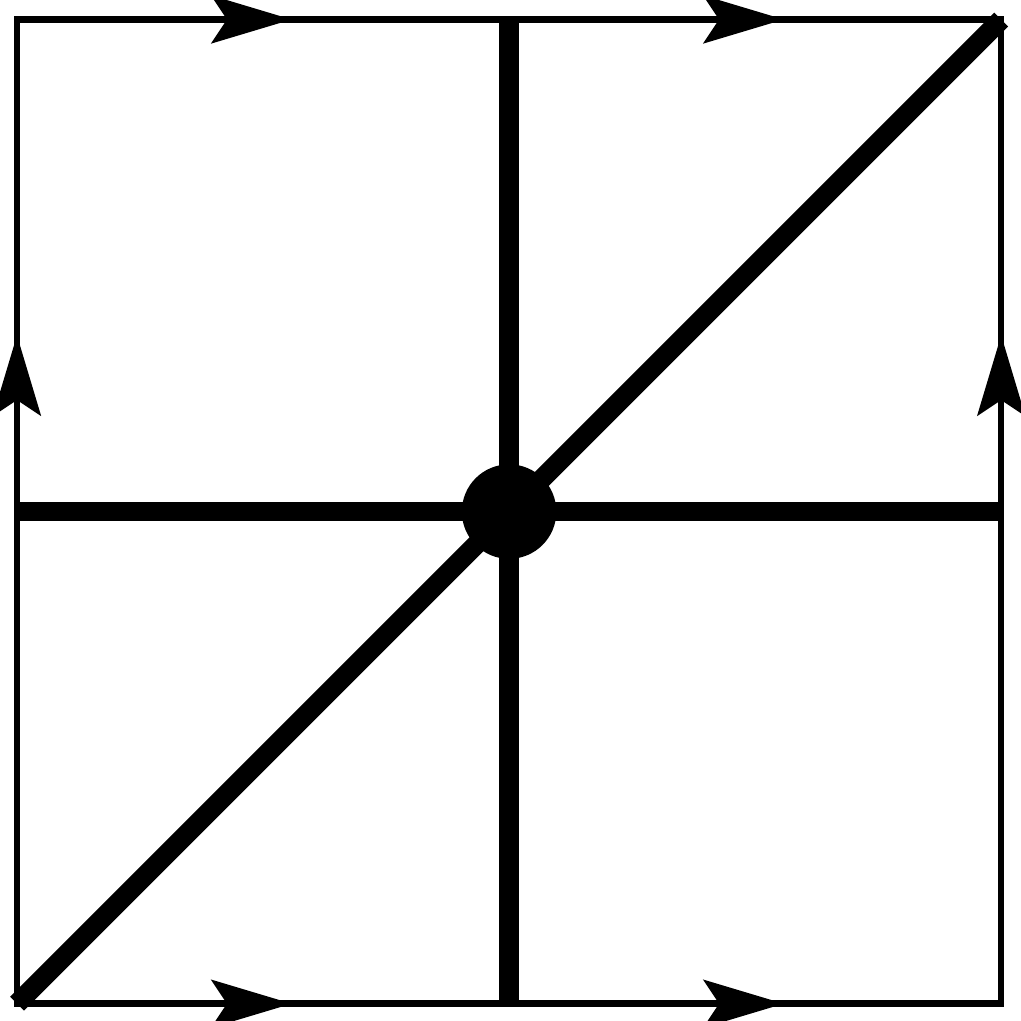}
\end{center}
\end{definition}

\begin{lemma}\label{Lem:ggl}
Let $G$ be a graph embedded in a surface $F$ of genus $g$, such that $G$ has no $3$--petal loops.
Let $k=|\ell_0|,\ n=|V(G)|-c(G)$, and $\gamma=\gamma(G)$.
Then $P_G(X,Y,U,V)$ has the following coefficient:
$$ \gamma\, UV^{g-1}X^nY^k.$$
\end{lemma}
\begin{proof}
As in the proof above, it suffices to prove the case $k=0$, so we
can assume that all loops in $G$ are homologically non-trivial.
We now determine all possible $H\subset G$ that can contribute to the term $UV^{g-1}X^n$ in $P_G(X,Y,U,V)$.
Due to the substitution $x = X-1$ and $y= Y-1$, we need to consider $H\subset G$ with weight $uv^{g-1}x^iy^j$.
Since $i\leq n$, the factor $X^n$ implies that $H$ can contribute to the term $UV^{g-1}X^n$ only if $i=n$.
Hence, $c(H)=|V(G)|$ so that $H\subset \ell(G)$ with weight $uv^{g-1}x^ny^j$.

Let $H'\subset G'$ be the reduced graph of $H$, as in \refdef{parallel}.
Let $\H'$ be the regular neighborhood of $H'$ in $F$.
The condition that $G$ has no $3$--petal loops implies that $G'$ and hence $H'$ have no $3$--petal loops.
By \cite[Equation~(4.7)]{Krushkal}, 
$$ k(H')+g(F)+g(\H')-g(F-\H') = b_1(H').$$
The factor $UV^{g-1}$ implies that $g(\H')=1$ and $g(F-\H')=g(F)-1$.
Thus, $k(H')=b_1(H')-2=|E(H')|-2$.
Since $g(\H')=1$, the condition that $H'$ has no $3$--petal loops now implies $k(H')=0$, so that $|E(H')|=2$.
So the only possible $H'\subset G'$ are the subgraphs $\{e_1\cup e_2\}\subset \ell(G')$ such that $g(\H')=1$. 
Therefore, if $H\subset G$ contributes to the term $UV^{g-1}X^n$ in $P_G(X,Y,U,V)$, then $H'$ is a pair of genus-generating loops.

Let $\{e_1'\cup e_2'\}\subset \ell(G')$ be a pair of genus-generating loops, and suppose for $I=1,2$, $G$ has $m_I$ parallel loops in the edge class $e_I'$.
Let $H_{i,j}\subset G$ denote the subgraph with $i$ loops (resp. $j$ loops) in the edge class $e_1'$ (resp. $e_2'$), which has weight $uv^{g-1}x^{n}y^{(i-1)+(j-1)}$.
As in \refeqn{lambda-sum}, summing over the weights of all $H_{i,j}\subset G$, we get the following contribution to $p_G(x,y,u,v)$:
\begin{equation}\label{Eqn:ggl-sum}
  \displaystyle{\sum_{\substack{1\leq i\leq m_1\\1\leq j\leq m_2}} {m_1 \choose i}{m_2 \choose j} uv^{g-1}x^{n}y^{(i-1)+(j-1)} =
    \frac{uv^{g-1}x^{n} }{y^2}((1+y)^{m_1}-1)((1+y)^{m_2}-1)}.
\end{equation}
Thus, for every pair of genus-generating loops in $G'$, its edge class
in $G$ contributes the expression \refeqn{ggl-sum} to $p_G(x,y,u,v)$.
As in \refeqn{lambda-sum2}, with the substitution $x = X-1$ and $y=
Y-1$, the expression \refeqn{ggl-sum} simplifies to
\[ UV^{g-1}X^{n}(1+Y^2 +\ldots +Y^{m_1-1})(1+Y^2 +\ldots +Y^{m_2-1}) + O(X^{n-1}). \]
Every pair of genus-generating loops in $G'$ contributes such an expression to
$P_G(X,Y,U,V)$, so if $k=|\ell_0|=0$, the coefficient on $UV^{g-1}X^{n}$ in
$P_G(X,Y,U,V)$ is $\gamma(G)$.
\end{proof}

\section{The homological twist number}\label{Sec:twist}

In this section, we introduce the homological twist number
$\tau_F(D)$, which counts sets of homologically twist-equivalent
crossings.  In contrast, the usual twist number $t_F(D)$, defined in
\cite[Definition~2.4]{KP}, counts \emph{twist regions} (maximal
strings of bigons) of $D$ on $F$.  Every twist region contributes one
homological twist to $\tau_F(D)$, but some crossings of $D$ which are
in distinct twist regions can be homologically twist-equivalent.  An
important advantage of \refdef{twist} below is that $\tau_F(D)$ is
invariant for any reduced alternating surface link diagram $D$,
without the need for $D$ to be twist-reduced.

\begin{figure}
\centerline{\includegraphics[height=82pt]{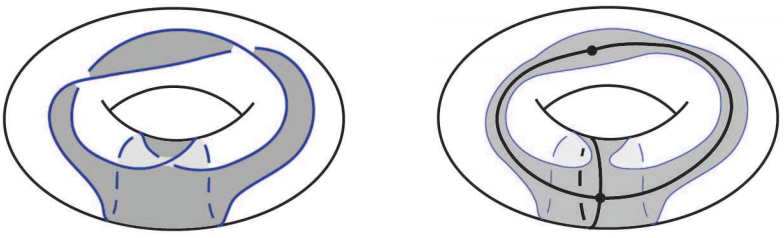}}
\caption{An alternating link diagram (left) and its Tait graph $G_A$ (right) on the torus \cite[Figure~5]{Krushkal}.}
\label{Fig:Tait}
\end{figure}

\begin{definition}\label{Def:Tait}
Let $D$ be a reduced alternating surface link diagram on $F$.  Fix a
checkerboard coloring on $D$.  Let $G_A$ (resp. $G_B$) be the Tait
graph (i.e., checkerboard graph) of $D$ on $F$, whose edges correspond
to crossings of $D$, and whose vertices correspond to shaded
(resp. unshaded) regions of $F-D$, such that $G_A$ and $G_B$ are dual
graphs on $F$.  See \reffig{Tait}.  Note that the Tait graph of a
reduced alternating surface link diagram may contain loops, but only
homologically non-trivial ones.  Let $G'_A$ and $G'_B$ be the
\emph{reduced Tait graphs} obtained by deleting all but one edge in
each set of parallel edges in $G_A$ and $G_B$, as in
\refdef{parallel}.
\end{definition}

See \reffig{many_cycles} for several examples of different kinds of cycles in the Tait graph on the surface $F$.

\begin{figure}
\centerline{\includegraphics[height=95pt]{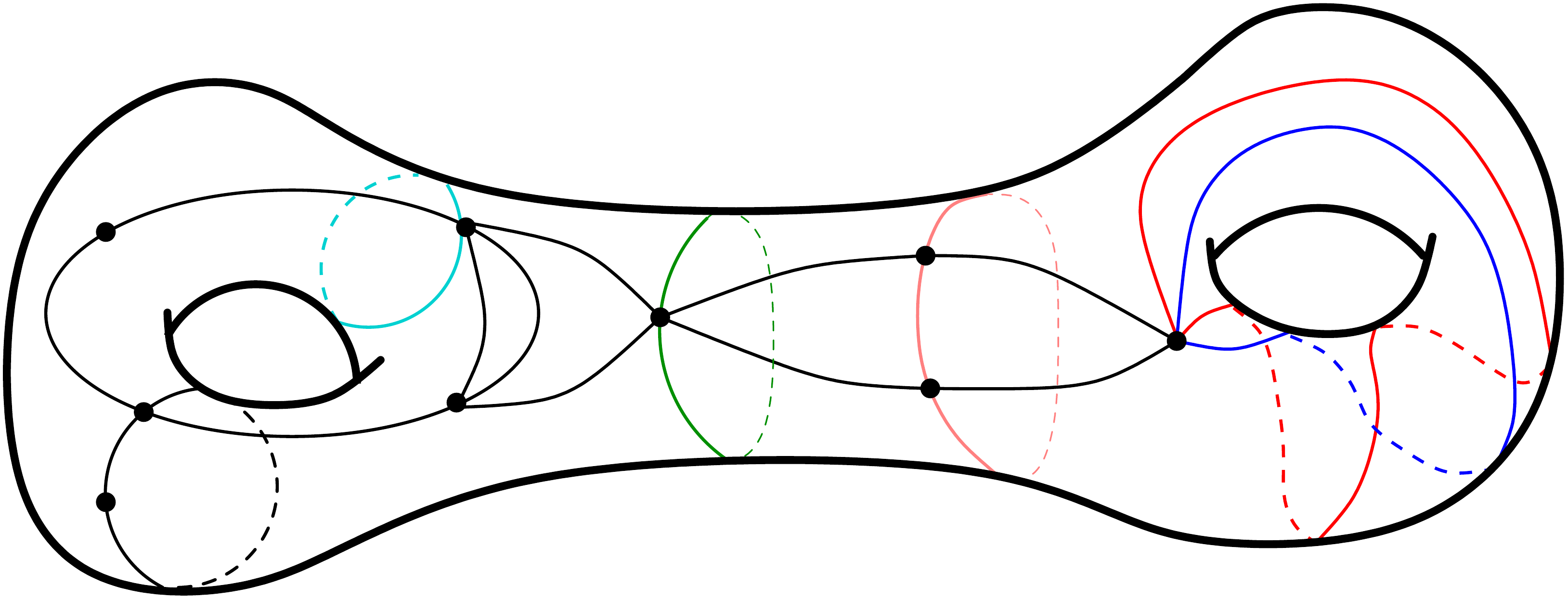}}
\caption{Different kinds of cycles in the Tait graph are shown in different colors. From left to right: loop giving the diagram representativity 2 (cyan), nugatory crossing (green), null-homologous 2-cycle (pink), genus-generating loops with representativity 4 (blue and red).  }
\label{Fig:many_cycles}
\end{figure}

\begin{definition}\label{Def:twist}
Recall, two edges in $G$ are parallel if they are homologous on $F$.
Two crossings of $D$ are \emph{homologically twist-equivalent} if
their corresponding edges are parallel in either $G_A$ or $G_B$.  The
\emph{homological twist number} $\tau_F(D)$ is defined as the number
of homological twist-equivalence classes of crossings of $D$.  Thus,
each homological twist corresponds to one set of parallel edges in
$G_A$ or $G_B$, which is one edge in $G'_A$ or $G'_B$.  
\end{definition}

\begin{figure}
\centerline{\includegraphics[height=90pt]{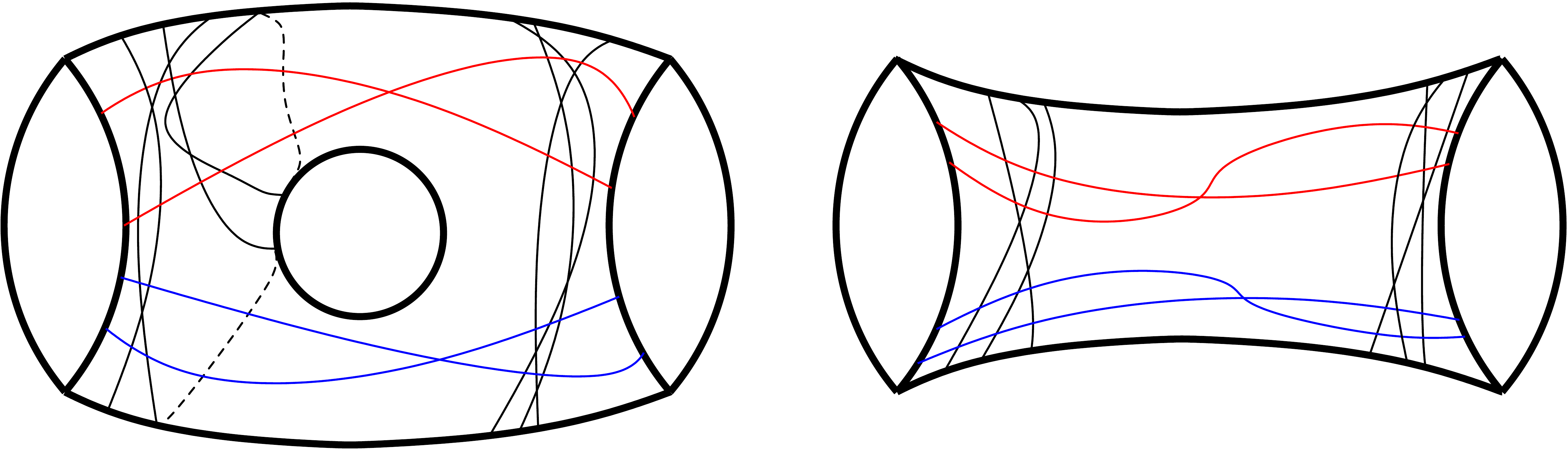}}
\caption{Two alternating link diagrams are projected on $F$, partly
  shown.  In both cases, the red crossing and the blue crossing are
  homologically twist-equivalent.  One Tait graph has homologous loops
  (left) or a null-homologous $2$--cycle (right).  Neither pair of
  crossings forms a twist region on $F$.}
\label{Fig:tau-crossings}
\end{figure}

See \reffig{tau-crossings} for two examples of homologically twist-equivalent crossings of $D$ on $F$, which do not form a twist region on $F$.

\begin{proposition}\label{Prop:twist}
If $t_F(D)$ denotes the twist number, as in \cite[Definition~2.4]{KP}, of a strongly reduced, twist-reduced WGA diagram, then
$$  \tau_F(D)\leq t_F(D)\leq 2\tau_F(D).$$
Moreover, if $g(F)\leq 1$ or the representativity $r(D,F)\geq 5$, then $\tau_F(D)=t_F(D)$.
\end{proposition}
\begin{proof}
Let $G_A$ and $G_B$ be the Tait graphs of $D$ on $F$, which do not
contain loops since $D$ is strongly reduced.  A pair of edges in $G_A$
or $G_B$ is parallel if and only if they form a null-homologous
$2$--cycle.  If it bounds a disk $\Delta$ on $F$, then the hypothesis
that $D$ is twist-reduced, as in \cite[Definition~2.5]{KP}, implies
that $\Delta$ or a disk in $F-\Delta$ contains a twist region of $D$,
which is the same as a homological twist-equivalence class of
crossings of $D$.  Thus, the two definitions of twist number agree in
this case.

On the other hand, suppose the null-homologous $2$--cycle bounds a
subsurface $F'\subset F$ which is not a disk, so it forms an essential
separating curve on $F$.  Hyperbolicity precludes both vertices from
being $2$--valent, but if one vertex is $2$--valent, then $D$ has a
bigon on $F$ and the two crossings are homologically twist-equivalent.
So the two definitions of twist number agree in this case as well.

However, if neither vertex is $2$--valent, then the two crossings are
homologically twist-equivalent, but are not part of a twist region
because $D$ is twist-reduced.  Moreover, this discrepancy occurs for
every essential null-homologous $2$--cycle without $2$--valent
vertices in $G_A$ or $G_B$.
This proves the inequality.

Finally, an essential null-homologous $2$--cycle in $G_A$ or $G_B$
bounds a compressing disk of $F$, and intersects the diagram $D$ in
$4$ points.  If $g(F)\leq 1$ or $r(D,F)\geq 5$, then neither $G_A$ nor $G_B$ admit
such a $2$--cycle.  In the remaining cases, $\tau_F(D)=t_F(D)$.
\end{proof}

\section{The Jones-Krushkal polynomial}\label{Sec:Jones-Krushkal}

In \cite{Krushkal}, Krushkal defined a homological Kauffman bracket
derived from his 4-variable polynomial $p_G(x,y,u,v)$, and proved the
invariance of a two-variable generalization of the Jones polynomial
for links in thickened surfaces.  We will use a later variant
$J_K(t,z)$, called the reduced Jones-Krushkal polynomial, which was
introduced by Boden and Karimi \cite{Boden-Karimi}.  Following
\cite{Krushkal}, it is proved in \cite{Boden-Karimi} that $J_K(t,z)$
is an invariant of oriented links under isotopy and diffeomorphism of
the thickened surface.

\begin{figure}
\centerline{\includegraphics[height=85pt]{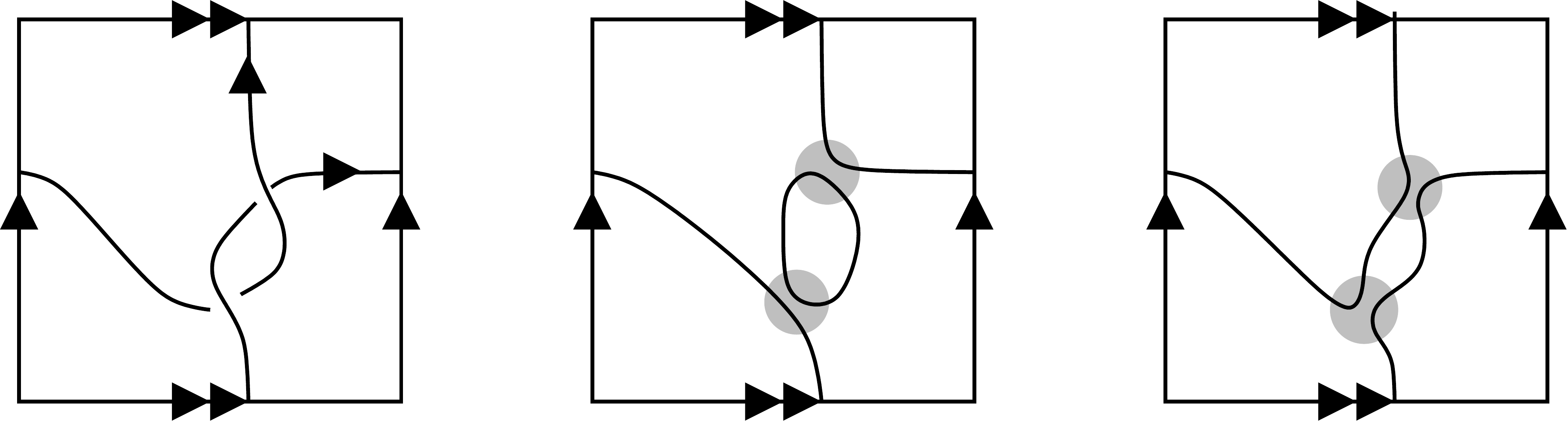}}
\caption{For $D$ on the torus (left), states $s_A$ (middle) and $s_B$ (right) are shown.
  Here, $|s_A|=2,\, |s_B|=1,\, r(s_A)=r(s_B)=1,\, k(s_A)=1,\, k(s_B)=0$.}
\label{Fig:states}
\end{figure}

We briefly recall the homological Kauffman bracket due to Krushkal
\cite{Krushkal}.  Let $F$ be a closed orientable surface of genus $g$.
Let $K$ be a link in $F\times I$, with a link diagram $D$ on $F$.
Suppose that $D$ has $c$ crossings, each of which can be resolved by
an $A$--smoothing or $B$--smoothing.  A \emph{state} $s$ of $D$ is a
collection of simple closed curves on $F$ that results from smoothing
each crossing of $D$.  See \reffig{states}.  Let $a(s)$ and $b(s)$ be
the number of $A$ and $B$--smoothings, and let $|s|$ be the number of
closed curves in $s$.  Let $s_A$ and $s_B$ denote the all--$A$ and
all--$B$ states of $D$, so that for the Tait graphs $G_A$ and $G_B$,
we have $|V(G_A)|=|s_A|$ and $|V(G_B)|=|s_B|$.  Let $n=|V(G_A)|-1$ and
$N=|V(G_B)|-1$.  Define
\begin{align*}
  k(s)=\dim(\text{kernel}(i_*\from H_1(s)\to H_1(F))),\\
  r(s)=\dim(\text{image}(i_*\from H_1(s)\to H_1(F))),
\end{align*}
where $i\from s\to F$ is the inclusion map.  We call $r(s)$ the
\emph{homological rank} of $s$, so that $k(s)+r(s)=|s|$.  The
homological Kauffman bracket is defined as follows:
$$ \langle D\rangle_F = \sum\nolimits_s
A^{(a(s)-b(s))}(-A^{-2}-A^2)^{k(s)}z^{r(s)}.$$ To recover the usual
Kauffman bracket for a classical diagram $D$, we set $z=-A^{-2}-A^2$
and divide by one factor of $-A^{-2}-A^2$.  To obtain the
Jones-Krushkal polynomial, which was the original link invariant
defined in \cite{Krushkal}, we normalize by the writhe as usual,
$(-A)^{-3w(D)}\langle D\rangle_F$, and set $A=t^{-1/4}$.

If $D$ is checkerboard colorable, then $[K]=0$ in $H_1(F\times I)$ by
\cite{Boden-Karimi}, so it follows that $k(s)\geq 1$ for every state
$s$ of $D$.  So we can instead use the following version of the
Jones-Krushkal polynomial due to Boden and Karimi:

\begin{definition}[\cite{Boden-Karimi}]\label{Def:Jones-Krushkal}
Let $K$ be an oriented link in $F\times I$, represented by a
checkerboard-colorable link diagram $D$ on $F$. The \emph{reduced Jones-Krushkal polynomial} is defined by
$$ J_K(t,z)=(-1)^{w(D)}t^{3w(D)/4}\sum\nolimits_s t^{(b(s)-a(s))/4}(-t^{-1/2}-t^{1/2})^{(k(s)-1)}z^{r(s)}.$$
\end{definition}

The reduced Jones-Krushkal polynomial specializes to the usual Jones
polynomial $V_K(t)$ by setting $z=-t^{-1/2}-t^{1/2}$.  Any classical
diagram will have $r(s)=0$ for all states, so that $J_K(t,z)=V_K(t)$
for every classical link $K$.  However, there exist alternating
virtual knots with $V_K(t)=1$ but non-trivial $J_K(t,z)$.

By \cite[Theorem~6.1]{Krushkal} for non-split $D$, we obtain $\langle D\rangle_F$ from
$P_{G_A}(X,Y,U,V)$ as follows:
$$ \langle D\rangle_F (A,z) = A^{(2g+2n-c)}d\, z^g P_{G_A}\left(-A^{-4},-A^4,\frac{A^2}{z},\frac{1}{A^2z}\right).$$
With the additional normalization as in \refdef{Jones-Krushkal}, we obtain $J_K(t,z)$ by
\begin{equation}\label{Eqn:J-P}
  J_K(t,z)=(-1)^w t^{(3w-2g-2n+c)/4} z^g P_{G_A}\left(-t,-t^{-1},\frac{1}{z\sqrt{t}},\frac{\sqrt{t}}{z}\right).
\end{equation}

Recall the definition of \emph{genus-generating loops} and \emph{$3$--petal loops} from \refdef{petal}.

\begin{definition}\label{Def:GGL}
For a reduced alternating diagram $D$ on $F$, let $\ell(G_A')$ and $\ell(G_B')$ be the subgraphs of loops in the reduced Tait graphs $G_A'$ and $G_B'$.
Define
\begin{align*}
  \gamma(D) &=\#\{\,\{e_1,e_2\}\subset\ell(G_A')\ |\ g(\H(e_1\cup e_2))>0\}, \\
  \bar\gamma(D) &=\#\{\,\{e_1,e_2\}\subset\ell(G_B')\ |\ g(\H(e_1\cup e_2))>0\}.
\end{align*}
\end{definition}

\begin{theorem}\label{Thm:Jones-Krushkal}
For a closed orientable surface $F$ of genus $g\geq 0$, let $K$ be a
non-split oriented link in $F\times I$ that admits a reduced
alternating diagram $D$ on $F$, such that neither of its Tait graphs has $3$--petal loops.
Let
$$ \lambda=|\ell(G_A')|,\ \bar\lambda=|\ell(G_B')|,\ \mu=b_1(G'_A-\ell(G_A')),\ \bar\mu=b_1(G'_B-\ell(G_B')),\ \gamma=\gamma(D),\ \bar\gamma=\bar\gamma(D).$$
Then 
\begin{equation}\label{Eqn:twist}
  \tau_F(D) = b_1(G_A') + b_1(G_B') -2g = \lambda + \mu +  \bar\lambda + \bar\mu -2g
\end{equation}
and the reduced Jones-Krushkal polynomial $J_K(t,z)$ has the following coefficients:
\begin{equation}\label{Eqn:JK}
 (-1)^{(w+n)}t^{\frac{3w+2n+c}{4}}\left((-1)^c t^{(g-c)}\left(\bar\lambda z t^{\frac{1}{2}} - (\bar\mu-\bar\gamma) t\right) - (\mu-\gamma) t^{-1} + \lambda z t^{-\frac{1}{2}}\right),
\end{equation}
where $c$ and $w$ are the crossing number and writhe of $D$, and $n=|V(G_A)|-1$.
\end{theorem}

We prove \refthm{Jones-Krushkal} after the following corollary, which
is important for \refthm{Intro}.  Recall that $D$ is strongly reduced
when neither $G_A$ nor $G_B$ has loops, so in particular,
$\gamma(D)=\bar\gamma(D)=0$.  In addition, $\gamma(D)=\bar\gamma(D)=0$
implies that neither Tait graph of $D$ has $3$--petal loops.

\begin{corollary}\label{Cor:Jones-Krushkal}
  If $D$ is a reduced alternating diagram on $F$, such that
  $\gamma(D)=\bar\gamma(D)=0$, then $\tau_F(D)$ is a link invariant of $K$ in $F\times I$.
\end{corollary}

\begin{proof}
For $g(F)=0$, the twist number is a link invariant by the proof of the
Tait flyping conjecture in \cite{menasco-thist:alternating}, so we may
assume $g(F)\geq 1$.  By \cite{Boden-Karimi}, $J_K(t,z)$ is an invariant
of $K$ in $F\times I$.  Thus, by \refthm{Jones-Krushkal}, $\tau_F(D)$
is a link invariant when $\gamma(D)=\bar\gamma(D)=0$, and the terms in
\refeqn{JK} are distinct terms in $J_K(t,z)$.

The terms in \refeqn{JK} coincide when $(-1)^c t^{(g-c)}=\pm t^{-1}$
or $\pm t^{-2}$; i.e., when $c=g+1$ or $c=g+2$.  As $D$ is cellularly
embedded, $c=|V_A|+|V_B|+2g-2$ with $|V_A|,|V_B|\geq 1$, which allows
only the cases: $(g,c)\in\{(1,2),(1,3),(2,4)\}$.  Moreover, both
$c=g+1$ and $c=g+2$ imply that either $|V_A|=1$ or $|V_B|=1$.  So one
Tait graph $G$ consists of only loops, and as $D$ is reduced
alternating, these loops are homologically non-trivial.

Let $H\subset G$.  For $[\partial\H]$ in $H_1(F)$, let $\Lambda(H)=\dim([\partial\H])$.
By \cite[Equation~(5.5)]{Krushkal}, 
$$ g(\H)+g(F-\H)+\Lambda(H)=g(F). $$
Since $D$ is cellularly embedded, then so is $G$.
Thus, for $H=G$ consisting of homologically non-trivial loops, we have $g(F-\H)=\Lambda(H)=0$.
Hence, $g(\H)>0$, which implies that at least one pair of loops in $G$ must be
genus-generating loops, which are excluded by the condition $\gamma(D)=\bar\gamma(D)=0$.

Thus, when $\gamma(D)=\bar\gamma(D)=0$, the terms in \refeqn{JK} are
distinct terms in $J_K(t,z)$.
\end{proof}

The proof of \refcor{Jones-Krushkal} relies on the condition $\gamma(D)=\bar\gamma(D)=0$, but it may not be necessary.

\begin{question}
  If $D$ is a reduced alternating diagram on $F$, is $\tau_F(D)$ a link invariant of $K$ in $F\times I$?
\end{question}

\begin{proof}[{\bf Proof of \refthm{Jones-Krushkal}}]
If $g=0$, then $D$ is a classical link diagram.  In this case,
$\lambda=\bar\lambda=0$ since loops in its Tait graph can only come
from nugatory crossings, so $\gamma=\bar\gamma=0$.  For classical links,
$J_K(t,z)=V_K(t)$, so now both \refeqn{twist} and \refeqn{JK} follow
from \cite{DL}.
  
To prove \refeqn{twist} for $g>0$, we extend the argument in \cite{DL}
to links in thickened surfaces.  Let
$G_A=(V_A,E_A),\ G'_A=(V_A,E_A'),\ G_B=(V_B,E_B),\ G'_B=(V_B,E_B')$.
Since $G_A$ and $G_B$ are dual graphs on $F$, $|E_A|=|E_B|$ and
$|V_A|+|V_B|=|E_A| + 2-2g$.  The homological twist number $\tau_F(D)$
counts sets of homologically twist-equivalent crossings, which we can
count using sets of parallel edges in $G_A$ and $G_B$, as follows:
  \begin{align*}
    \tau_F(D) &= |E_A|-(|E_A|-|E_A'|)-(|E_B|-|E_B'|)\\
    &= |E_A'|+|E_B'|-|E_A|\\
    &= |E_A'|+|E_B'| - (|V_A|+|V_B|-2+2g)\\
    &= (|E_A'|-|V_A|+1) + (|E_B'|-|V_B|+1) -2g\\
    &= b_1(G_A') + b_1(G_B') -2g\\
    &= \lambda + \mu +  \bar\lambda + \bar\mu -2g.
  \end{align*}

We now prove \refeqn{JK} for $g>0$.  Let $P_{G_A}(X,Y,U,V)$ be as in
\refthm{Krushkal}, with $G=G_A$.  By duality \cite[Theorem~3.1]{Krushkal},
$P_{G_B}(X,Y,U,V)=P_{G_A}(Y,X,V,U)$.  Therefore, by \refthm{Krushkal},
$\lambda, \bar\lambda, \mu, \bar\mu$ are exactly the coefficients of
the following terms of $P_{G_A}(X,Y,U,V)$:
\begin{equation}\label{Eqn:Krushkal-terms}
  \mu\,V^g X^{n-1} + \lambda\, V^{g-1} X^n + \bar\mu\, U^g Y^{N-1} + \bar\lambda\, U^{g-1} Y^N,
\end{equation}
where $n=|V_A|-1$ and $N=|V_B|-1$.  Using $\chi(F)=|V_A|+|V_B|-c$, we have $n+N=c-2g$.

Let $\pi(X^\alpha Y^\beta U^i V^j)\in\Z[t^{\pm 1/2},z]$ denote
the term in $J_K(t,z)$ obtained from $X^\alpha Y^\beta U^i V^j$ by
the substitutions in \refeqn{J-P}.
We evaluate each term in \refeqn{Krushkal-terms}: 
\begin{align*}
  {\textstyle\pi(V^g X^{n-1})} &={\textstyle (-1)^w t^{\frac{3w-2g-2n+c}{4}} z^g \left(\frac{\sqrt{t}}{z}\right)^g (-t)^{n-1} = (-1)^{(w+n)}\,t^{\frac{3w+2n+c}{4}} (-t^{-1})},\\
  {\textstyle\pi(V^{g-1} X^n)} &={\textstyle (-1)^w t^{\frac{3w-2g-2n+c}{4}} z^g \left(\frac{\sqrt{t}}{z}\right)^{g-1}\!\! (-t)^n = (-1)^{(w+n)}\,t^{\frac{3w+2n+c}{4}} (z t^{-\frac{1}{2}})},\\
  {\textstyle\pi(U^g Y^{N-1})} &={\textstyle (-1)^w t^{\frac{3w-2g-2n+c}{4}} z^g \left(\frac{1}{z\sqrt{t}}\right)^g (-t)^{2g+n-c+1}=(-1)^{(w+n)}\,t^{\frac{3w+2n+c}{4}} (-1)^c t^{(g-c)} (-t)},\\
  {\textstyle\pi(U^{g-1} Y^N)} &={\textstyle (-1)^w t^{\frac{3w-2g-2n+c}{4}} z^g \left(\frac{1}{z\sqrt{t}}\right)^{g-1}\!\! (-t)^{2g+n-c} = (-1)^{(w+n)}\,t^{\frac{3w+2n+c}{4}} (-1)^c t^{(g-c)} (z t^{\frac{1}{2}})}.
\end{align*}  
This verifies that the terms in \refeqn{JK} come from the corresponding terms in \refeqn{Krushkal-terms}.
We now find the other terms in $P_{G_A}(X,Y,U,V)$ that overlap with these terms in $J_K(t,z)$.

For the $\mu$--term, suppose $\pi(X^\alpha Y^\beta U^i V^j)=\pm\pi(V^g X^{n-1})$.  
Since the RHS has no $z$ factor, it follows that $i+j=g$.
From exponents on $t$, we have
$$ \alpha-\beta-i/2+j/2=g/2+n-1 \implies \alpha+j+1=\beta+g+n.$$
If $\alpha=n-\kappa$ for some integer $\kappa\geq 0$, then
$$ n-\kappa+j+1=\beta+(i+j)+n \geq 0 \implies \kappa=0 \text{\ or \ } \kappa=1.$$
If $\alpha=n$ then $\beta+i=1$, so $\beta,i\in\{0,1\}$.  If $\alpha=n-1$ then $\beta+i=0$, so $\beta=i=0$.
We are left with only three possibilities:
\begin{align*}
  \alpha=n-1,\ \beta=0,\ i=0,\ j=g &\implies V^g X^{n-1} \\
  \alpha=n,\ \beta=0,\ i=1,\ j=g-1 &\implies UV^{g-1} X^n \\
  \alpha=n,\ \beta=1,\ i=0,\ j=g &\implies V^g X^nY 
\end{align*}
We already know $\mu\,V^g X^{n-1}$ is in $P_{G_A}(X,Y,U,V)$.  Since
$G_A$ does not have $3$--petal loops, we can apply \reflem{ggl} to see
that $UV^{g-1}X^n$ has coefficient $\gamma$ in $P_{G_A}(X,Y,U,V)$.  As
a term in $J_K(t,z)$, $\pi(V^g X^{n-1})=-\pi(UV^{g-1} X^n)$ because
$X^n$ and $X^{n-1}$ contribute opposite signs, so we call it the
$(\mu-\gamma)$--term in $J_K(t,z)$.  For the final case above, we
claim that $V^g X^nY$ cannot be a term in $P_{G_A}(X,Y,U,V)$.  Suppose
there exists $H\subset G_A$ whose weight contributes to $V^g X^nY$.
As in the proof of \reflem{ggl}, the factor $X^n$ implies $H\subset
\ell(G_A)$.  Because $D$ is reduced alternating on $F$, all loops in
$G_A$ are homologically non-trivial.  The factor $Y$ implies that $H$
has weight with a factor $y^k$ for $k>0$, so $H$ must contain
$3$--petal loops, which are excluded by hypothesis.  Thus, $V^g X^nY$
cannot be a term in $P_{G_A}(X,Y,U,V)$.  With the cases exhausted, we
see that no other terms in $P_{G_A}(X,Y,U,V)$ besides $V^g X^{n-1}$
and $UV^{g-1}X^n$ contribute to the $(\mu-\gamma)$--term in
$J_K(t,z)$.

For the $\bar\mu$--term, we can use duality \cite[Theorem~3.1]{Krushkal}:  $P_{G_A}(X,Y,U,V)=P_{G_B}(Y,X,V,U)$.
If $\pi(X^\alpha Y^\beta U^i V^j)=\pm\pi(U^g Y^{N-1})$, the argument above for the dual graph $G_B$ again implies only three possibilities:
\begin{align*}
  \alpha=0,\ \beta=N-1,\ i=g,\ j=0 &\implies U^g Y^{N-1} \\
  \alpha=0,\ \beta=N,\ i=g-1,\ j=1 &\implies U^{g-1}V Y^N \\
  \alpha=1,\ \beta=N,\ i=g,\ j=0 &\implies U^g XY^N 
\end{align*}  
By the same arguments on the dual graph, for $D$ reduced alternating, only $U^g Y^{N-1}$ and $U^{g-1}V Y^N$ are terms in $P_{G_A}(X,Y,U,V)$.
Therefore, no other terms in $P_{G_A}(X,Y,U,V)$ besides these terms contribute to the $(\bar\mu-\bar\gamma)$--term in $J_K(t,z)$. 

For the $\lambda$--term, suppose $\pi(X^\alpha Y^\beta U^i V^j)=\pm\pi(V^{g-1} X^n)$.
Since the RHS has a $z$ factor, it follows that $i+j=g-1$.
From exponents on $t$, we have
$$ \alpha-\beta-i/2+j/2=(g-1)/2+n \implies \alpha=\beta+i+n.$$
If $\alpha=n-\kappa$ for some integer $\kappa\geq 0$, then
$$ n-\kappa=\beta+i+n \geq 0 \implies \beta=i=\kappa=0.$$
This leaves only one possibility:
$$  \alpha=n,\ \beta=0,\ i=0,\ j=g-1 \implies V^{g-1} X^n.$$
Therefore, no other terms in $P_{G_A}(X,Y,U,V)$ besides $V^{g-1} X^n$ contribute to the $\lambda$--term in $J_K(t,z)$. 
For the $\bar\lambda$--term, we can use a similar argument or use duality again.

This completes the proof of \refeqn{JK}.
\end{proof}

\begin{lemma}\label{Lem:extremal}
For $K$ in $F\times I$ as in \refthm{Jones-Krushkal}, only the terms $V^gX^n$
and $U^gY^N$ of $P_{G_A}(X,Y,U,V)$ contribute the extremal terms of
$J_K(t,1)$, which has span $(c-g)$.
\end{lemma}
\begin{proof}
By \cite[Theorem~2.9]{Boden-Karimi}, and dividing by one factor of
$-A^{-2}-A^2$ for the reduced polynomial, the span of $J_K(t,1)$ is
exactly $(c-g)$.  We now identify the subgraphs of $G_A$ that
contribute the two extremal terms of $J_K(t,1)$.  By \refeqn{J-P}, the
term in $P_{G_A}(X,Y,U,V)$ which contributes the highest $t$--degree
term of $J_K(t,1)$ has the highest $X$--degree and highest
$V$--degree.  Namely, the unique spanning subgraph $H_0$ in $G_A$ with
an empty edge set has weight $v^gx^n$.  Similarly, $H=G_A$ has weight
$u^gy^N$, which contributes the the lowest $t$--degree term of
$J_K(t,1)$.  Thus, $P_{G_A}(X,Y,U,V)$ has the terms $V^gX^n$ and
$U^gY^N$, which contribute the extremal terms of $J_K(t,1)$.

We claim that no other terms of $P_{G_A}(X,Y,U,V)$ contribute the
extremal terms of $J_K(t,1)$.  Suppose there exists $H\subset G_A$
whose weight also contributes to $V^g X^n$.  As in the proof of
\reflem{ggl}, the factor $X^n$ implies $H$ has weight with factor
$x^n$ and $H\subset \ell(G_A)$.  Thus, $H$ has weight $v^gx^ny^k$ for
$k\geq 0$.  Because $D$ is reduced alternating on $F$, all loops in
$G_A$ are homologically non-trivial.  If $k>0$ then $H$ must contain
$3$--petal loops, which are excluded by hypothesis.  Thus, only $H_0$
contributes the term $V^gX^n$ in $P_{G_A}(X,Y,U,V)$.  The argument
for $H=G_A$ follows by duality \cite[Theorem~3.1]{Krushkal},
$P_{G_B}(X,Y,U,V)=P_{G_A}(Y,X,V,U)$.
\end{proof}

\begin{proof}[{\bf Proof of \refthm{Intro}}]
Since $D$ is strongly reduced, $\lambda=\bar\lambda=0$.
Thus, by \refeqn{twist},
$$ \tau_F(D) = \mu + \bar\mu - 2g,$$
which is a link invariant of $K$ in $F\times I$ by \refcor{Jones-Krushkal}.

We claim that the $\mu$ and $\bar\mu$ terms in \refeqn{JK} with
$\lambda=\bar\lambda=0$ and $\gamma=\bar\gamma=0$ are exactly the
sub-extremal terms of $J_K(t,0)$.  By \reflem{extremal}, only the
terms $V^gX^n$ and $U^gY^N$ of $P_{G_A}(X,Y,U,V)$ contribute the
extremal terms of $J_K(t,1)$, which has span $(c-g)$.  The $\mu$ and
$\bar\mu$ terms in \refeqn{Krushkal-terms} differ from $V^gX^n$ and
$U^gY^N$, and in $J_K(t,1)$ they have span $(c-g-2)$ by \refeqn{JK},
so they are the sub-extremal terms of $J_K(t,1)$.
Moreover, by \refeqn{J-P} neither the extremal terms nor the $\mu$ and $\bar\mu$ terms
have a $z$ factor in $J_K(t,z)$.  Thus, the $\mu$ and $\bar\mu$ terms
in \refeqn{JK} are exactly the sub-extremal terms of $J_K(t,0)$.  This
proves the first part of \refthm{Intro}.

By \refprop{twist},
$$ \tau_F(D)\leq t_F(D)\leq 2\tau_F(D).$$
The volume bounds in \refthm{Intro} now follow from
\cite[Theorem~1.4]{KP}.  Since essential null-homologous cycles occur
only for $g\geq 2$, the bounds for $g=1$ are the same.  Since
$\tau_F(D)\leq t_F(D)$, the lower bound for $g\geq 2$ is the same.
Since $t_F(D)\leq 2\tau_F(D)$, the upper bound for $g\geq 2$ must be
doubled.
\end{proof}  

\section{Examples}

Below we provide data to confirm \refthm{Krushkal} and \refthm{Jones-Krushkal} for three virtual links.

\subsection*{Example 1}

\begin{figure}
\centerline{ \includegraphics[height=100pt]{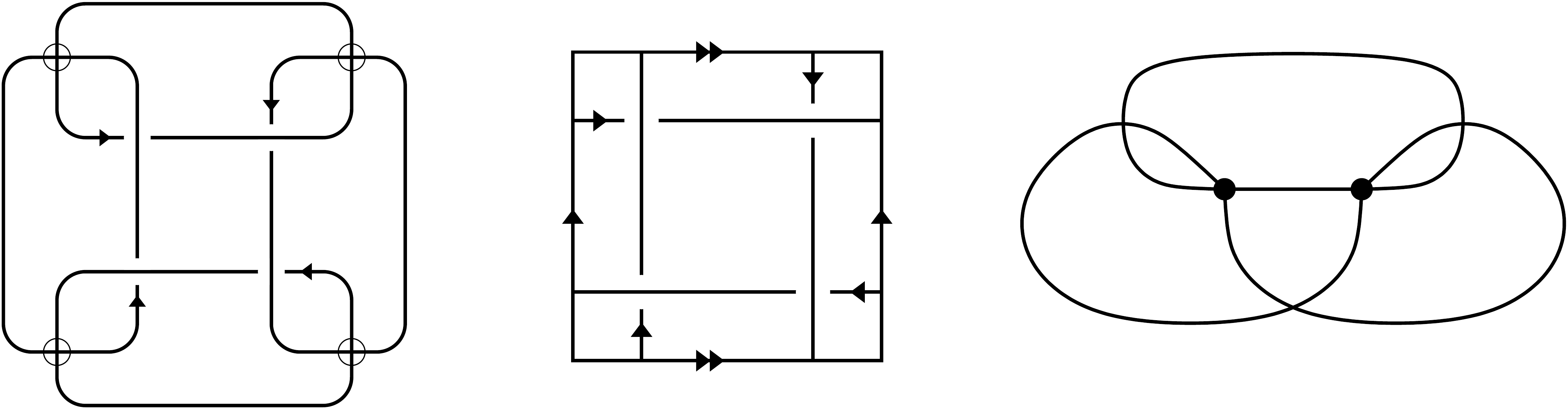}}
\caption{The $2\times 2$ square weave on the torus as a virtual link, its diagram $D$ with $\tau_F(D)=4$,
  and self-dual Tait graphs $G_A=G_B$ shown as a ribbon graph.}
\label{Fig:BK15}
\end{figure}

The $4$--component virtual link $K_1$ shown in \reffig{BK15} is also discussed in \cite[Example~3.10]{Boden-Karimi}.
For its $2\times 2$ square weave diagram $D$ on the torus, $\tau_F(D)=4$.
We have the following data from this diagram:
\begin{align*}
&  g=1,\ \mu= 3,\ \lambda= 0,\ \gamma= 0,\ \bar\mu= 3,\ \bar\lambda= 0,\ \bar\gamma= 0,\ c=4,\ w=-4,\ n=1,\ N=1.\\ \\
& \text{Eqn}~\refeqn{twist}:\ \tau_F(D) = \lambda + \mu +  \bar\lambda + \bar\mu -2g = 4 \\ \\
&  \text{Eqn}~\refeqn{Krushkal-terms}:\  \mu\,V^g X^{n-1} + \lambda\, V^{g-1} X^n + \bar\mu\, U^g Y^{N-1} + \bar\lambda\, U^{g-1} Y^N = 3V+3U \\ \\
&  P_{G_A}(X,Y,U,V) = VX+6+UY+\ \framebox{$3V+3U$} \\ \\
&  \text{Eqn}~\refeqn{JK}:\ -\bar\lambda(zt^{-4}) + (\bar\mu-\bar\gamma)t^{-7/2} + (\mu-\gamma)t^{-5/2} - \lambda(z t^{-2})= 0+3t^{-7/2} + 3t^{-5/2}+0 \\ \\
&  J_K(t,z) =  -t^{-9/2} +\ \framebox{$3t^{-7/2} + 3t^{-5/2}$} - t^{-3/2} + \left(6zt^{-3}\right) \\ \\
\end{align*}

These results agree with \refthm{Krushkal} and \refthm{Jones-Krushkal}.
Note that $\lambda=\bar\lambda=0$, so we can compute
$\tau_F(K_1)$ directly from the sub-extremal coefficients of
$J_K(t,0)$.  As discussed in \cite{ckp:gbal}, the hyperbolic volume of
$T^2\times I-K_1$ is $4\voct$, which is within the bounds of
\refthm{Intro} for $\tau_F(K_1)=4$, although $D$ has representativity
$2$.

\vfill
\pagebreak

\subsection*{Example 2}

\begin{figure}
\centerline{\includegraphics[height=150pt]{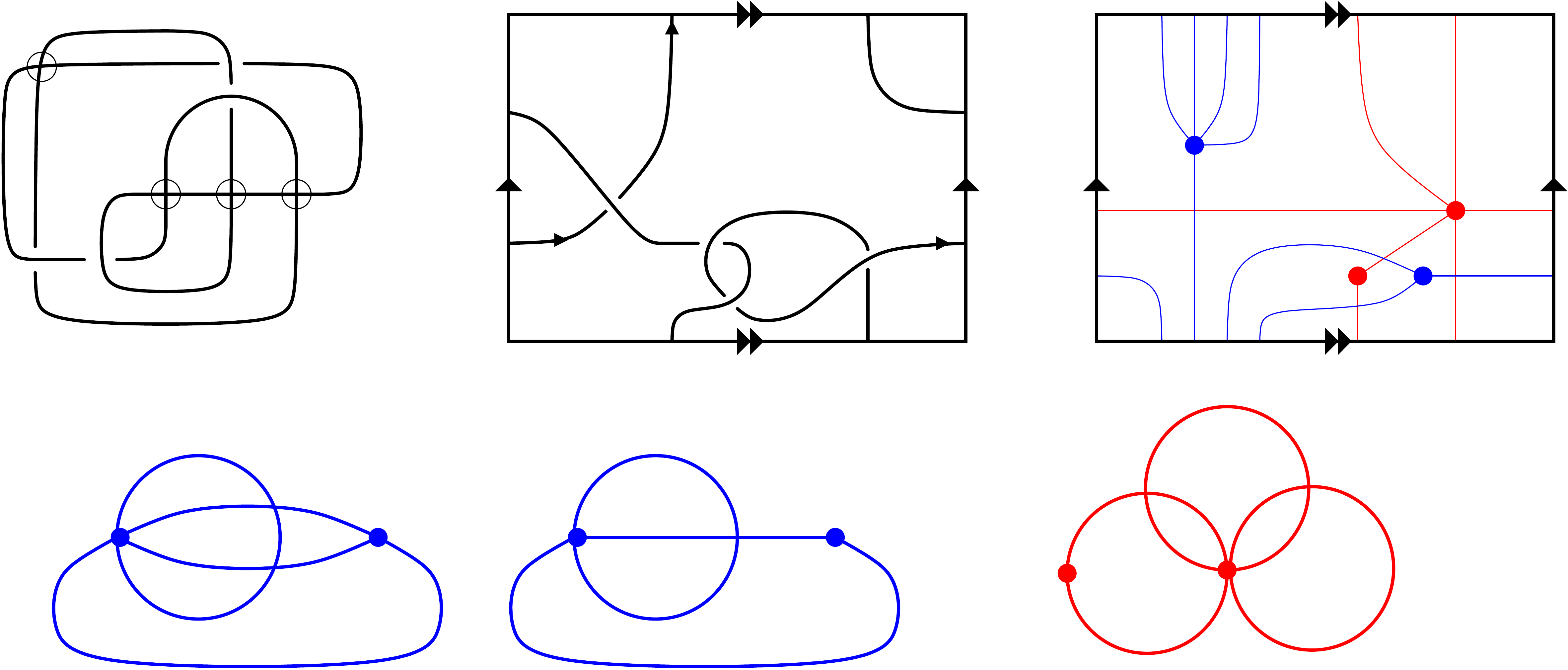}}
\caption{First row, left to right: Virtual knot 4.106, its diagram $D$ on the torus with $\tau_F(D)=3$, and its Tait graphs $G_A$ (red) and $G_B$ (blue) on the torus.
  Second row, left to right, shown as ribbon graphs: Tait graph $G_B$ and its reduction $G_B'$ (blue), and $G_A$ (red) which is already reduced.  Note the pair of genus-generating loops in $G_A$.}
\label{Fig:4-106}
\end{figure}

The virtual knot $K_2=4.106$ is shown in \reffig{4-106}, with a diagram $D$ shown on the torus.
We have the following data from this diagram:
\begin{align*}
&  g=1,\ \mu= 1,\ \lambda= 2,\ \gamma= 1,\ \bar\mu= 1,\ \bar\lambda= 1,\ \bar\gamma= 0,\ c=4,\ w=-2,\ n=1,\ N=1.\\ \\
& \text{Eqn}~\refeqn{twist}:\ \tau_F(D) = \lambda + \mu +  \bar\lambda + \bar\mu -2g = 3 \\ \\
&  \text{Eqn}~\refeqn{Krushkal-terms}:\  \mu\,V^g X^{n-1} + \lambda\, V^{g-1} X^n + \bar\mu\, U^g Y^{N-1} + \bar\lambda\, U^{g-1} Y^N = V + 2X + U + Y  \\ \\
&  P_{G_A}(X,Y,U,V) = UX+UY+VX+\framebox{$V+2X+U+Y$}+2\\ \\
&  \text{Eqn}~\refeqn{JK}:\ \bar\lambda(-zt^{-5/2}) + (\bar\mu-\bar\gamma)t^{-2} + (\mu-\gamma)t^{-1} +  \lambda(-z t^{-1/2})= -zt^{-5/2} + t^{-2} + 0 - 2zt^{-1/2} \\ \\
&  J_K(t,z) =  -t^{-3} + \framebox{$t^{-2}$} -1 + \left(\framebox{$-zt^{-5/2}$}+2zt^{-3/2} \ \framebox{$-2zt^{-1/2}$}\right)\\ \\
\end{align*}

These results agree with \refthm{Krushkal} and \refthm{Jones-Krushkal}.
Note that one of the coefficients in \refeqn{JK} is zero because
$\mu=\gamma=1$.  In this case, $\tau_F(K_2)$ cannot be computed directly
from the coefficients of $J_K(t,z)$.  Also, note that if we set
$z=-t^{-1/2}-t^{1/2}$, then $J_K(t,z) = 1$, so the virtual knot
$4.106$ has trivial Jones polynomial.

\vfill
\pagebreak

\subsection*{Example 3}

\begin{figure}
 \centerline{\includegraphics[height=190pt]{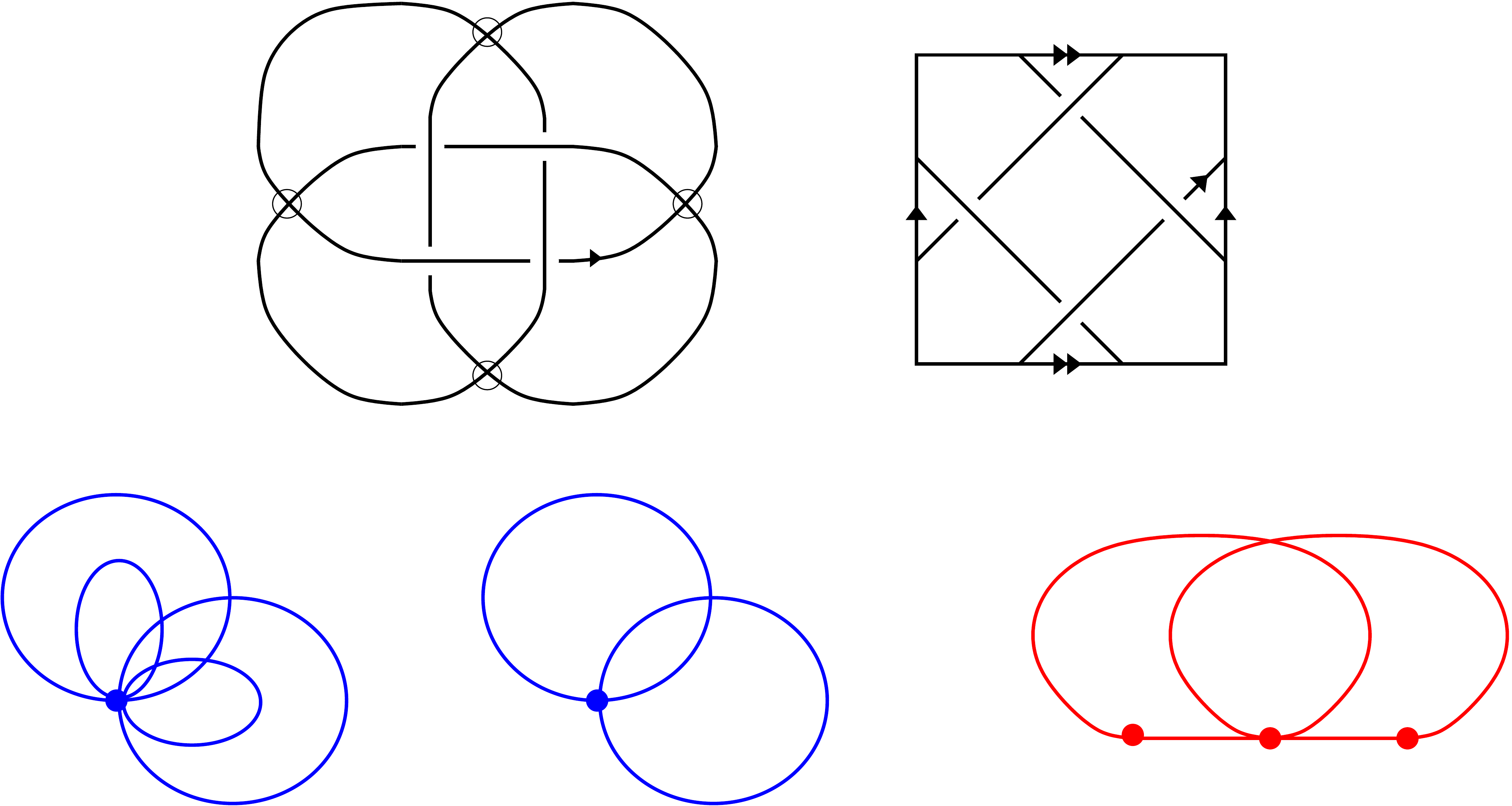}} 
\caption{First row, left to right: Virtual knot 4.105, and its diagram $D$ on the torus with $\tau_F(D)=2$.
  Second row, left to right, shown as ribbon graphs: Tait graph $G_B$ and its reduction $G_B'$ (blue), and $G_A$ (red) which is already reduced.
  Note the pair of genus-generating loops in $G_B'$.}
\label{Fig:8_18}
\end{figure}

The virtual knot $K_3=4.105$ is shown in \reffig{8_18}, with a diagram $D$ shown on the torus.
From the diagram on the torus, we can see $\tau_F(D)=2$, but it is less apparent from the virtual link diagram which evokes the knot $8_{18}$.
We have the following data from this diagram:
\begin{align*}
&  g=1,\ \mu= 2,\ \lambda= 0,\ \gamma= 0,\ \bar\mu= 0,\ \bar\lambda= 2,\ \bar\gamma= 1,\ c=4,\ w=-4,\ n=2,\ N=0.\\ \\
& \text{Eqn}~\refeqn{twist}:\ \tau_F(D) = \lambda + \mu +  \bar\lambda + \bar\mu -2g = 2 \\ \\
&  \text{Eqn}~\refeqn{Krushkal-terms}:\  \mu\,V^g X^{n-1} + \lambda\, V^{g-1} X^n + \bar\mu\, U^g Y^{N-1} + \bar\lambda\, U^{g-1} Y^N = 2VX+2 \\ \\
&  P_{G_A}(X,Y,U,V) = VX^2+U+V+2X+\ \framebox{$2VX+2$} \\ \\
&  \text{Eqn}~\refeqn{JK}:\ \bar\lambda(zt^{-7/2}) - (\bar\mu-\bar\gamma)t^{-3} - (\mu-\gamma)t^{-2} +  \lambda(z t^{-3/2})= 2zt^{-7/2}+ t^{-3} -2t^{-2} +0 \\ \\
&  J_K(t,z) =  t^{-4} +\ \framebox{$t^{-3} -2t^{-2}$} + t^{-1} + \left(\framebox{$2zt^{-7/2}$}\ -2zt^{-5/2}\right)\\ \\
\end{align*}

These results agree with \refthm{Krushkal} and \refthm{Jones-Krushkal}.
Note that because $\bar\gamma=1,\ \tau_F(K_3)$ cannot be computed
directly from the coefficients of $J_K(t,z)$.

\bibliographystyle{amsplain}
\bibliography{references}

\end{document}